\documentclass[11pt]{article}
\usepackage{amsmath, amssymb, amscd, amsthm, amsfonts}
\usepackage{graphicx}
\usepackage{hyperref}
\usepackage{tikz,tikz-cd}
\usepackage[dvipsnames]{xcolor}
\usepackage[centertableaux,nobaseline]{ytableau}
\usepackage{mathtools}
\usepackage[shortlabels]{enumitem}

\DeclareMathOperator{\lft}{left}
\DeclareMathOperator{\up}{up}
\DeclareMathOperator{\id}{id}

\newcommand{\stack}[5]{ 
    \begin{tikzpicture}[scale=0.6]
        \draw[thick] (-0.5,0) -- (1,0) -- (1,-2) -- (2,-2) -- (2,0) -- (3.5,0);
        \node[fill = white, draw = white] at (2.8,.5) {$#5$};
        \node[fill = white, draw = white] at (1.5,-1.5) {$#4$};
        \node[fill = white, draw = white] at (1.5,-.8) {$#3$};
        \node[fill = white, draw = white] at (1.5,-.1) {$#2$};
        \node[fill = white, draw = white] at (.2,.5) {$#1$};
    \end{tikzpicture}
}

\newcommand{\arrow}{ 
    \!\!\!\!\!\!\!
    \begin{tikzpicture}[scale=0.6]
        \node at (0,0) {};
        \node at (0,0.7) {$\longrightarrow$};
    \end{tikzpicture}
    \!\!\!\!\!\!\!
}

\title{Motzkin Numbers Count 2-Stack-Sortable Permutations Ending in Their Least Entry}
\author{Ryota Inagaki \and Michael Luo}
\date{September 28, 2026}

\newtheorem{theorem}{Theorem}[section]

\newtheorem{lemma}[theorem]{Lemma}
\newtheorem{prop}[theorem]{Proposition}

\theoremstyle{remark}

\newtheorem*{remark*}{Remark}

\theoremstyle{definition}

\newtheorem{example}[theorem]{Example}

\begin{document}

\maketitle

\begin{abstract}
We prove the following conjecture of Zhang (arXiv:2604.10779, Conjecture 6.1): for $n \geq 0$, the number of $2$-stack-sortable permutations of $\{0,1,\dots,n\}$ ending in $0$ is the $n$th Motzkin number. By Zhang's result, there is a bijection between $2$-stack-sortable permutations ending in their least element and standard composition tableaux of width at most $2$. We then show bijectively that there are an equal number of these and standard Young tableaux of width at most $3$, which are known to be counted by the Motzkin numbers.
\end{abstract}

\section{Introduction}\label{section-introduction}

In 1990, West \cite{west} introduced the stack-sorting map $s$, which processes a permutation through a sorted stack by repeatedly pushing in an element if it is smaller than the topmost element in the stack and popping out the stack's topmost element otherwise. Figure \ref{fig:westmap} shows the action of $s$ on $3142$.

\begin{figure}[h]
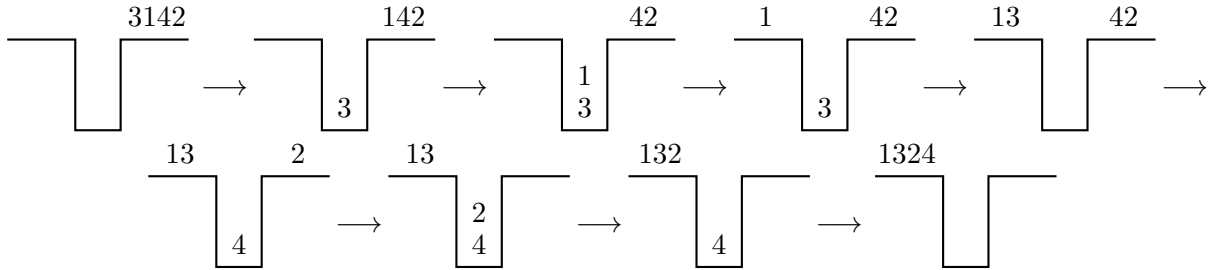

    \begin{center}
        \stack{}{}{}{}{3142}\arrow
        \stack{}{}{}{3}{142}\arrow
        \stack{}{}{1}{3}{42}\arrow
        \stack{1}{}{}{3}{42}\arrow
        \stack{13}{}{}{}{42}\arrow
        \stack{13}{}{}{4}{2}\arrow
        \stack{13}{}{2}{4}{}\arrow
        \stack{132}{}{}{4}{}\arrow
        \stack{1324}{}{}{}{}
    \end{center}
    \caption{For the permutation $\pi = 3142 \in S_4$, we have $s(3142) = 1324$.}
    \label{fig:westmap}
\end{figure}

For a permutation $\pi$, its \emph{stack-sorting depth} $\langle \pi\rangle$ is the smallest number of applications of $s$ needed to sort $\pi$. As $s(3142) = 1324$ and $s(1324) = 1234$, the stack-sorting depth of $3142$ is $2$.

The problem of enumerating permutations $\pi$ with $\langle \pi\rangle \le t$ is open for general $t$. However, for $t$ small, there are some results. Let $S_n$ be the set of permutations of $\{1,\dots,n\}$. We say a permutation is $t$-stack-sortable if its stack-sorting depth is at most $t$.

Knuth \cite{knuth} showed the number of $1$-stack-sortable permutations in $S_n$ is the $n$th Catalan number $C_n$, and Zeilberger \cite{zeilberger} showed the number of $2$-stack-sortable permutations in $S_n$ is $\frac2{(n + 1)(2n + 1)}\binom{3n}n$. Defant \cite{3-stack-sortable} then found a polynomial-time algorithm to compute the number of $3$-stack-sortable permutations in $S_n$, but for depths $t \ge 4$, the problem is still open.

Let $S_n' = \{\pi\mathbf 0 \mid \pi \in S_n\}$ be the set formed by appending $0$ to each permutation in $S_n$. Knuth \cite{knuth} showed the $s$-sortable permutations are exactly the ones avoiding the pattern $231$. Therefore, the only $s$-sortable permutation in $S_n'$ is the decreasing one defined by $\pi = \pi_1\dots \pi_{n + 1}$ with $\pi_i = n + 1 - i$. Zhang \cite{tableau} introduced stack-sorting tableaux, which associate each $\pi \in S_n'$ with an antistandard composition tableau. Using these tableaux, he found for each $t$ a polynomial-time algorithm enumerating the $t$-stack-sortable permutations in $S_n'$. He conjectured $2$-stack-sortable permutations in $S_n'$ are counted by the Motzkin numbers, which is our main result. The Motzkin numbers (OEIS A001006 \cite{oeis}) count, among other things, the number of ways to draw any number of disjoint chords connecting $n$ points on a circle.
\begin{theorem}
    \label{thm:motzkin}
    The number of $2$-stack-sortable permutations in $S_n'$ is the $n$th Motzkin number $M_n$.
\end{theorem}

The paper is organized as follows. In Section \ref{sec:preliminaries}, we introduce notation for permutations, as well as composition and Young tableaux. We also introduce Zhang's stack-sorting tableaux, which are necessary for showing $2$-stack-sortable permutations in $S_n'$ are in bijection with standard composition tableaux of width at most $2$. In Section \ref{sec:proof}, we present the bijective arguments proving Theorem \ref{thm:motzkin}.

\section{Preliminaries}\label{sec:preliminaries}
\subsection{Permutations}

A \emph{permutation} of a totally ordered set is a sequence containing each element in that set exactly once. Let the unique permutation of the empty set be $\varepsilon$. We denote the concatenation of two permutations $u$ and $v$ by $uv$. For a one-digit integer $n$, we denote the unique permutation of $\{n\}$ by the digit $n$ written in boldface. For example, $(6,7) = \mathbf{67}$. Let $\widetilde S_n$ be the set of permutations $\pi$ of $\{0,\dots,n\}$ such that the subpermutation of $\pi$ after $0$ is sorted. Note that $S_n' \subseteq \widetilde S_n$, as $\varepsilon$ is sorted.

\subsection{Composition Tableaux and Young Tableaux}

We review some definitions about tableaux. A \emph{composition} $\alpha$ of a positive integer $n$ is a sequence $(\alpha_1,\dots,\alpha_k)$ such that the $\alpha_i$ are positive integers adding to $n$. A \emph{partition} is a composition where the elements are in weakly decreasing order.

We use the English convention that the diagram of a composition $\alpha = (\alpha_1,\dots,\alpha_k)$ is a set of cells formed by stacking $k$ left-aligned rows of length $\alpha_1,\dots,\alpha_k$ on top of each other, with $\alpha_1$ on top and $\alpha_k$ on bottom. By convention, we identify a composition $\alpha$ with its diagram and refer to both as $\alpha$.

For a diagram $\alpha$, a \emph{tableau} is a function $\tau: \alpha \to \mathbb N$. The tableau is said to have \emph{shape} $\alpha$. If $\alpha$ is a partition, we say a tableau of shape $\alpha$ is a Young tableau. A tableau is said to be \emph{increasing} or \emph{decreasing} if its rows and columns are strictly increasing or strictly decreasing, respectively, where rows are read left-to-right and columns are read top-to-bottom. A tableau with $n$ cells is said to be \emph{standard} (resp.\ \emph{antistandard}) if it contains each element of $\{1,\dots,n\}$ exactly once and it is increasing (resp.\ decreasing). For example,
\[\begin{ytableau}
    5 & 3 \\ 4 \\ 2 & 1
\end{ytableau}\]
is an antistandard composition tableau with shape $(2,1,2)$. Let $\mathcal T_n$ be the set of antistandard composition tableaux on $n$ cells.


\subsection{Stack-Sorting Tableaux}

For a cell $c$ in a composition tableau, let $\lft(c)$ be the cell directly to the left of $c$ and $\up(c)$ be the cell above $c$ in its same column, possibly crossing a gap, if these are defined. By convention, we identify a cell with its entry. In the above example, $\lft(\up(1)) = 5$. The \emph{width} of a tableau or a diagram is the length of its longest row. For a cell $c$ in a composition tableau, let its \emph{hook-length} $h(c)$ be the number of cells that satisfy all of the following: strictly left of $c$, weakly above $c$, and strictly below $\up(\lft(c))$. If $\up(\lft(c))$ does not exist, we drop that condition. This should not be confused with the usual hook-length of a cell $c$ in a Young tableau; we never use the usual hook-length in this paper.

\begin{example}
The hook-length of $1$ in the tableau below is $6$, as shown below.
\[\begin{ytableau}11 & 9 & 5 \\ *(Salmon) 10 & *(Salmon) 7 \\ *(Salmon) 8 \\ *(Salmon) 6 & *(Salmon) 3 & *(Salmon) 2 & *(Apricot) 1 \\ 4 \end{ytableau}\]
\end{example}

The following result summarizes the theory developed in Sections 3 and 4 of \cite{tableau}.
\begin{prop}
    \label{pro:theory}
    There exists a map $\theta$ from $S_n'$ to $\mathcal T_n$ such that
    \begin{itemize}
        \item the preimage of every $T \in \mathcal T_n$ has $\prod_ch(c)$ elements, where the product is taken over all cells $c$ not in the first column of $T$, and
        \item for all $\pi \in S_n'$, the width of $\theta(\pi)$ is the stack-sorting depth of $\pi$.
    \end{itemize}
\end{prop}

For completeness, we provide a description of the map $\theta$ described in Proposition \ref{pro:theory}. We require the following lemma, whose proof is a specialization of the inductive step of the proof of Lemma 3.1 of \cite{disproof} to the case of $s$. We include the proof for completeness.

\begin{lemma}
    \label{lem:final-element}
    The set $\widetilde S_n$ is closed under $s$.
\end{lemma}

\begin{proof}
    Let $\pi \in \widetilde S_n$. We need to show the part of $s(\pi)$ after $0$ is increasing. Consider the moment immediately after $0$ is popped from the stack when $s$ is applied on $\pi$. Let the remainder of the input permutation be $r = (r_1,\dots,r_i)$ and the elements of the stack from top to bottom be $s = (s_1,\dots,s_j)$. Now $r$ and $s$ are both increasing. If $s_1 < r_1$, then $s_1$ is popped. Otherwise, $s_1 > r_1$, so $r_1$ is pushed in, then immediately popped. In either case, the least element of $r$ and $s$ is added to the end of the output permutation. Iterating $i + j$ times moves the elements of $r$ and $s$ to the end of $s(\pi)$ in increasing order.
\end{proof}

We now find a way to represent the stack-sorting map on $\widetilde S_n$. For fixed $\pi \in \widetilde S_n$, let $\pi = \rho \mathbf 0r$ for $r$ increasing. Note that $\pi$ is sorted if and only if $\rho$ is empty. More generally, $\pi \in \widetilde S_n$ is determined only by $\rho$, therefore, it is only necessary to track $\rho$ when analyzing the action of $s$ on $\pi$.

We will construct the stack-sorting tableau $T_\pi$ corresponding to $\pi \in \widetilde S_n$, which has $\lvert \rho\rvert$ cells and whose entries are exactly those of $\rho$. Set $\rho_i$ to be the portion of $s^{i - 1}(\pi)$ left of $0$.

For a permutation $\pi$, let $M(\pi)$ be the subpermutation of right-to-left maxima of $\pi$. Observe that $M(\rho_i)$ is exactly the set of elements that are in the stack when $0$ is pushed in during the $i$th iteration of $s$. Therefore, the $M(\rho_i)$ are the elements that are moved after $0$, and hence are removed from $\rho_i$, during the $i$th iteration of $s$.

Let the $i$th column of $T_\pi$ consist of the elements of $M(\rho_i)$ in order, with the leftmost element of $M(\rho_i)$ being the topmost entry in the $i$th column of $T_\pi$. For all cells but the first column, we require $c$ to be between $\up(\lft(c))$ and $\lft(c)$ in $s^{i - 2}(\pi)$, which uniquely determines the position of $c$. If $\up(\lft(c))$ is undefined, we just require $c$ to be left of $\lft(c)$. Let $\theta$ send $\pi \in \widetilde S_n$ to its stack-sorting tableau $T_\pi$.

\begin{example}
    For the permutation $\mathbf{8395714260} \in S_9'$, we have
    \begin{align*}
        \rho_1 &= \mathbf{83\textcolor{red}95\textcolor{red}7142\textcolor{red}6} \\
        \rho_2 &= \mathbf{3\textcolor{red}8\textcolor{red}512\textcolor{red}4} \\
        \rho_3 &= \mathbf{\textcolor{red}31\textcolor{red}2} \\
        \rho_4 &= \mathbf{\textcolor{red}1} \\
        \rho_5 &= \varepsilon,
    \end{align*}
    where the red elements of $\rho_i$ are the elements of $M(\rho_i)$. Then $T_\pi$ is the following.
    \[\begin{ytableau}
        9 & 8 & 3 \\
        7 & 5 \\
        6 & 4 & 2 & 1
    \end{ytableau}\]
    The tableau $T_\pi$ contains each element of $\{1,\dots,9\}$ exactly once. Also, $T_\pi$ is antistandard and has width $4 = \langle \pi\rangle$.
\end{example}

\section{Proof of Zhang's Conjecture}
\label{sec:proof}

Let $W'_2(n)$ be the set of $2$-stack-sortable permutations in $S_n'$. Let $X_n'$ be the set of antistandard composition tableaux with $n$ cells and width at most $2$, $X_n$ be the set of standard composition tableaux with $n$ cells and width at most $2$, and $Y_n$ be the set of standard Young tableaux with $n$ cells and width at most $3$. We will show the chain of equalities
\[\lvert W'_2(n)\rvert = \lvert X_n'\rvert = \lvert X_n\rvert = \lvert Y_n\rvert.\]
Conjugation yields a bijection between $Y_n$ and standard Young tableaux with $n$ cells and at most three rows. As $M_n$ counts standard Young tableaux with $n$ cells and at most three rows \cite{motzkinsyt}, this finishes. The proof is divided into a series of claims.
\begin{lemma}
    \label{lem:first}
    $\lvert W'_2(n)\rvert = \lvert X_n'\rvert$.
\end{lemma}

\begin{lemma}
    \label{lem:second}
    $\lvert X_n'\rvert = \lvert X_n\rvert$.
\end{lemma}

\begin{theorem}
    \label{thm:third}
    $\lvert X_n\rvert = \lvert Y_n\rvert$.
\end{theorem}

We begin by proving the first two, as they are relatively simple.

\begin{proof}[Proof of Lemma \ref{lem:first}]
    We claim the map $\theta$ of Proposition \ref{pro:theory} is a bijection from $W_2'(n)$ to $X_n'$. By the second part of Proposition \ref{pro:theory}, if $\pi \in W_2'(n)$, then $\theta(\pi)$ is an antistandard composition tableau with $n$ cells and at most $2$ columns. Also, the hook-length of every cell $c$ in the second column of $\theta(\pi)$ is $1$, since the unique cell that is strictly left of $c$, weakly above $c$, and strictly below $\up(\lft(c))$ is the cell immediately to the left of $c$. Therefore, the first part of Proposition \ref{pro:theory} implies the preimage of every $T \in X_n'$ consists of exactly one element. Hence $\theta: W_2'(n) \to X_n'$ is a bijection and $\lvert W_2'(n)\rvert = \lvert X_n'\rvert$, as desired.
\end{proof}

\begin{proof}[Proof of Lemma \ref{lem:second}]
    The map replacing each entry $c$ with $n + 1 - c$ is an involution on the set of tableaux sending $X_n'$ to $X_n$, hence $\lvert X_n'\rvert = \lvert X_n\rvert$.
\end{proof}

We now work towards proving Theorem \ref{thm:third}. Let $T$ be an increasing composition tableau. We define the following if they exist.
\begin{itemize}
    \item Let the \emph{overhang} of $T$ be the highest two-cell row of $T$ that is immediately below a one-cell row. Let the entries of the overhang be $i(T)$ and $j(T)$ with $i(T) < j(T)$.
    \item Let $i'(T)$ be the least cell in the second column that is larger than $i(T)$.
    \item Let $c_3(T)$ be the bottommost cell in the third column. Let $c_2(T)$ be the largest cell smaller than $c_3(T)$ in the second column and $c_1(T)$ be the largest cell smaller than $c_2(T)$ in the first column. Let $c_2'(T)$ be the bottommost cell in the second column weakly above $c_1$.
    \item If $T$ has an overhang, let $d_2(T)$ be the largest cell in the second column smaller than $j$, if it exists. Otherwise, let $d_2(T)$ be the largest cell in the second column, if it exists.
\end{itemize}
We abbreviate $i(T)$, $j(T)$, $i'(T)$, $c_1(T)$, $c_2(T)$, $c_3(T)$, and $d_2(T)$ to $i$, $j$, $i'$, $c_1$, $c_2$, $c_3$, and $d_2$, respectively, if $T$ is clear from context.

A standard composition tableau $T$ is said to be \emph{intermediate} if all of the following conditions hold.
\begin{itemize}
    \item \textbf{Width}: $T$ has width at most $3$.
    \item \textbf{Shape}: there is no three-cell row of $T$ below a two-cell or one-cell row of $T$.
    \item \textbf{Position}: $T$ has width at most $2$, $T$ has no overhang, (P1) both $i > d_2$ and $j > c_3$, or (P2) $c_3 < i' < j$.
\end{itemize}
Note that the conditions (P1) and (P2) are mutually exclusive, since if $i' < j$, we have $i' \le d_2$ since the second column is increasing. Yet $i' > i$, so this would imply $i < d_2$. Also, to check for the position condition, showing $i' > c_3$ is sufficient. If $c_3$ does not exist, there is nothing to prove. Otherwise, $j > i$, so $i' \le j$, but if $i' = j$, we have $i > d_2$ and $j = i' > c_3$.

Let $Z_n$ be the set of all intermediate tableaux on $n$ cells. Observe that $X_n \subseteq Z_n$, since each element of $X_n$ has width at most $2$. Also, $Y_n \subseteq Z_n$, since each element of $Y_n$ has a partition shape and hence no overhang. We will construct a pair of inverse bijections $f: Z_n \setminus Y_n \to Z_n \setminus X_n$ and $g: Z_n \setminus X_n \to Z_n \setminus Y_n$.

Let $T \in Z_n \setminus Y_n$. If $T$ has no overhang, then there is no three-cell row of $T$ below a two-cell or one-cell row of $T$, by the shape condition, and there is no two-cell row of $T$ below a one-cell row of $T$. Therefore, $T$ has partition shape and $T \in Y_n$, which is impossible since $T \in Z_n \setminus Y_n$. Hence $T$ must have an overhang. We construct $f(T)$ by cases. We demonstrate each of these cases in Example \ref{exa:f}.
\begin{enumerate}[(a)]
    \item Suppose $i' = j$. Then move $i$ to the highest available position in the second column and $j$ to the highest available position in the third column.
    \item Suppose $i' < j$. Then move $i'$ to the highest available position in the third column and $i$ to the original position of $i'$. Then move $j$ to the highest available position in the second column.
\end{enumerate}
Note that if $T$ has width $3$, we use case (a) of $f$ if and only if (P1) holds and case (b) of $f$ if and only if (P2) holds.

\begin{example}
    \label{exa:f}
    We apply $f$ three times on the tableau below in sequence, where we use cases (a), (b), and (a) of $f$, in that order. The overhangs are highlighted in green.
    
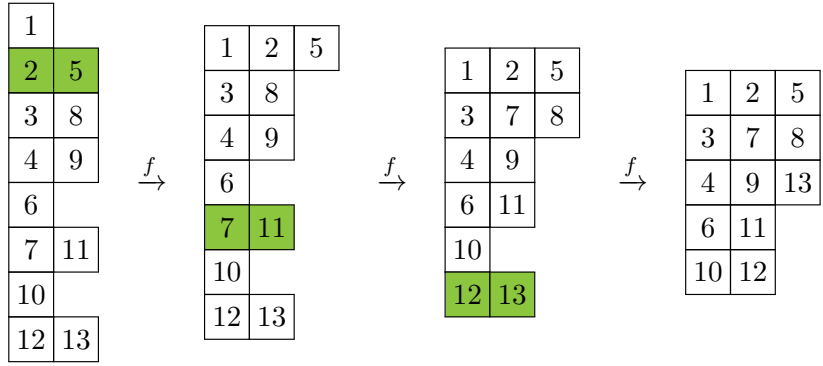
\[\begin{ytableau}
    1 \\ *(LimeGreen) 2 & *(LimeGreen) 5 \\ 3 & 8 \\ 4 & 9 \\ 6 \\ 7 & 11 \\ 10 \\ 12 & 13
    \end{ytableau} \quad \xrightarrow f \quad \begin{ytableau}
    1 & 2 & 5 \\ 3 & 8 \\ 4 & 9 \\ 6 \\ *(LimeGreen) 7 & *(LimeGreen) 11 \\ 10 \\ 12 & 13
    \end{ytableau} \quad \xrightarrow f \quad \begin{ytableau}
    1 & 2 & 5 \\ 3 & 7 & 8 \\ 4 & 9 \\ 6 & 11 \\ 10 \\ *(LimeGreen) 12 & *(LimeGreen) 13
    \end{ytableau} \quad \xrightarrow f \quad \begin{ytableau}
    1 & 2 & 5 \\ 3 & 7 & 8 \\ 4 & 9 & 13 \\ 6 & 11 \\ 10 & 12
    \end{ytableau}\]
\end{example}

We now check $f$ is well-defined. Let $T \in Z_n \setminus Y_n$
and let $F = f(T)$. We first check that $F$ is increasing.
The first column of $F$ is increasing since it is just the first column of $T$ with
$i(T)$ removed. In the second column, $i'(T)$ is replaced by
$i(T)$, while all other entries retain their order, so
the second column of $F$ is increasing.
The third column is obtained by appending $i'(T)$.
If $T$ has width $3$, $i'(T) > c_3(T)$, so the third column of $F$ is increasing.
Also, the rows of $F$ are increasing, since the cell immediately to the left of any second- or third-column cell in $F$ is both weakly left of it and weakly above it in $T$.

Now we verify $F$ is intermediate.
The width condition holds, as no cells are inserted into any columns beyond the third, and the shape condition holds, since the new cell in the third column is in the lowest available position.
We now verify the position condition. If $F$ has no overhang, there
is nothing to prove, so suppose otherwise. Now $i(F)$ was below $i(T)$ in the first column of $T$, so $i(F)>i(T)$. Also, every second-column entry of $T$ smaller than $i'(T)$ is smaller than
$i(T)$, so every second-column entry of $F$ greater than $i(F)$ is greater than $i'(T)$. Hence
\[
    i'(F)>i'(T)=c_3(F),
\]
so the position condition holds. Since the third column of $F$ is nonempty,
$f:Z_n\setminus Y_n\to Z_n\setminus X_n$ is well-defined.

We now construct the inverse $g: Z_n \setminus X_n \to Z_n \setminus Y_n$ of $f$. For any $T \in Z_n \setminus X_n$, $T$ has width $3$, so $c_3$ and $d_2$ exist. We construct $g(T)$ by cases. We demonstrate each of these cases in Example \ref{exa:g}.
\begin{enumerate}[(a)]
    \item If $d_2 < c_3$, move $d_2$ into the first column such that the first column stays increasing, creating a new row, then move $c_3$ to the position immediately to the right of $d_2$.
    \item Otherwise, if $c_1$ and $c_2$ are in the same row, insert $c_2$ into the first column such that it stays increasing (which is directly below $c_1$ by assumption), then move $c_3$ immediately to the right of $c_2$.
    \item Otherwise, insert $c_2$ into the first column such that it stays increasing, move $c_2'$ directly to the right of $c_2$, then move $c_3$ to the old position of $c_2$.
\end{enumerate}

\begin{example}
    \label{exa:g}
    We apply $g$ three times on the tableau below in sequence, where we use cases (a), (c), and (b) of $g$, in that order.
    
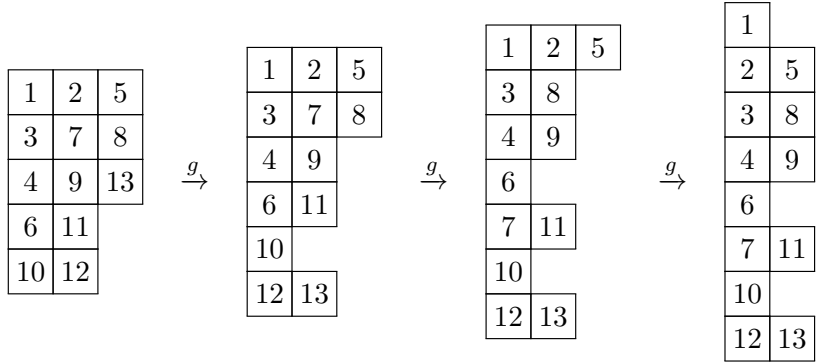
\[\begin{ytableau}
    1 & 2 & 5 \\ 3 & 7 & 8 \\ 4 & 9 & 13 \\ 6 & 11 \\ 10 & 12
    \end{ytableau} \quad \xrightarrow g \quad \begin{ytableau}
    1 & 2 & 5 \\ 3 & 7 & 8 \\ 4 & 9 \\ 6 & 11 \\ 10 \\ 12 & 13
    \end{ytableau} \quad \xrightarrow g \quad \begin{ytableau}
    1 & 2 & 5 \\ 3 & 8 \\ 4 & 9 \\ 6 \\ 7 & 11 \\ 10 \\ 12 & 13
    \end{ytableau} \quad \xrightarrow g \quad \begin{ytableau}
    1 \\ 2 & 5 \\ 3 & 8 \\ 4 & 9 \\ 6 \\ 7 & 11 \\ 10 \\ 12 & 13
    \end{ytableau}\]
    Note that this is exactly the reverse of the sequence in Example \ref{exa:f}. This holds since $f$ and $g$ are inverses, which we will show in Lemma \ref{lem:inverses}.
\end{example}

We now check that $g$ is well-defined. Let $T \in Z_n \setminus X_n$ and let $G = g(T)$. Observe that $G \not\in Y_n$, as $G$ always has an overhang. We check this, working by the cases of $g$.
\begin{enumerate}[(a)]
    \item Case (a). Then $i(G) = d_2(T)$ and $j(G) = c_3(T)$.
    \item Case (b). Then $i(G) = c_2(T)$ and $j(G) = c_3(T)$.
    \item Case (c). We claim that $c_2'(T) \le d_2(T)$. If $c_2'(T) > d_2(T)$, $c_1(T)$ is weakly lower than $i(T)$, hence $i(T) \le c_1(T) < c_2(T) \le d_2(T)$, so (P1) does not hold for $T$. Also, as $c_2(T) > i(T)$, we have $i'(T) \le c_2(T) < c_3(T)$, so (P2) does not hold for $T$ either, impossible. Hence in $G$, the row above the new overhang $i(G) = c_2(T)$ and $j(G) = c_2'(T)$ has the one cell $c_1(T)$.
\end{enumerate}
We now check that $G$ is increasing. The fact that the first column and rows of $G$ are increasing follows directly from the definition of $g$. We now check that the second column is increasing, working by the cases of $g$. In the following discussion, any cell names such as $c_3$ are taken with respect to $T$.
\begin{enumerate}[(a)]
    \item Case (a). If $T$ has no overhang, $c_3$ is the bottommost cell of the second column of $G$, so the second column stays increasing as $c_3 > d_2$. Suppose now $T$ has an overhang. (P2) for $T$ must not hold, since then $i' \le d_2$, so
    \[c_3 < i' \le d_2 < c_3.\]
    Since $T$ is intermediate, this implies (P1), or $i > d_2$ and $j > c_3$. Then $c_3$ is inserted into the second column above $j$, so the second column stays increasing as $j > c_3 > d_2$.
    \item Case (b). The action of $g$ on the second column here is to replace $c_2$ with $c_3$. We see $c_2 < c_3$, and if there are any cells below $c_2$ in the second column, $c_3$ is less than them, by the definition of $c_2$. Hence the second column remains increasing.
    \item Case (c). We have $c_2' > c_2$, since $c_1$ and $c_2$ are not in the same row and $c_2' \le d_2$. Now the action of $g$ on the order of the second column only replaces $c_2$ with $c_3$, which finishes since (as in the previous case) replacing $c_2$ with $c_3$ keeps the second column increasing.
\end{enumerate}
Now we check the conditions for intermediate tableaux. The width
condition is immediate. The definitions imply that $c_2$ and $d_2$
are weakly below $c_3$, and $c'_2$ is weakly below $c_2$. Thus all
changes occur in the original row of $c_3$ or below it, and the new
row is inserted below that row. Since $c_3$ is removed, the remaining
three-cell rows are above every shorter row. Hence the shape
condition holds.

We now verify the position condition for $G$, working by the cases of $g$. Suppose $c_3(G)$ exists, otherwise, there is nothing to prove. In the following discussion, $\up$ is taken with respect to $T$.
\begin{enumerate}[(a)]
    \item Case (a). Then (P1) holds for $G$, since
    \[i(G) = d_2(T) > \up(d_2(T)) = d_2(G) \quad \text{and} \quad j(G) = c_3(T) > \up(c_3(T)) = c_3(G).\]
    \item Case (b). Then (P1) holds for $G$, since
    \[i(G) = c_2(T) > \up(c_2(T)) = d_2(G) \quad \text{and} \quad j(G) = c_3(T) > \up(c_3(T)) = c_3(G).\]
    \item Case (c). Then $i'(G) = c_3(T) > \up(c_3(T)) = c_3(G)$, which suffices.
\end{enumerate}
Therefore, $g: Z_n \setminus X_n \to Z_n \setminus Y_n$ is well-defined.

\begin{lemma}
    \label{lem:inverses}
    The maps $f$ and $g$ are inverses.
\end{lemma}

\begin{proof}
    We need to show $f \circ g = g \circ f = \id$. We first show $f \circ g = \id$. Let $T \in Z_n \setminus X_n$. Let $F = f(T)$ and $G = g(T)$. We work via cases of $g$.
    \begin{enumerate}[(a)]
        \item Case (a). Then $i(G) = d_2(T)$ and $j(G) = c_3(T)$. Then case (a) of $f$ reverses the application of $g$ to $T$.
        \item Case (b). Similar to case (a), but $i(G) = c_2(T)$ instead of $d_2(T)$. Case (a) of $f$ reverses the application of $g$ to $T$ again.
        \item Case (c). Then $i(G) = c_2(T)$, $j(G) = c_2'(T)$, and $i'(G) = c_3(T)$. Since $i'(G) < j(G)$, case (b) of $f$ reverses the application of $g$ to $T$.
    \end{enumerate}
    Now we show $g \circ f = \id$. Let $T \in Z_n \setminus Y_n$. We work via cases of $f$.
    \begin{enumerate}[(a)]
        \item Case (a). Suppose there is exactly one $1$-cell row above the overhang of $T$, and the row immediately below the overhang of $T$ has $2$ cells. Then $c_1(F)$ is the cell above $i(T)$ in $T$, $c_2(F) = i(T)$, and $c_3(F) = j(T)$. Hence $c_1$ and $c_2$ are in the same row, and applying case (b) of $g$ reverses the application of $f$ to $T$. Otherwise, $d_2(F) = i(T)$ and $c_3(F) = j(T)$, and case (a) of $g$ reverses the application of $f$ to $T$.
        \item Case (b). Then $c_3(F) = i'(T)$, $c_2(F) = i(T)$, $c_1(F)$ is the cell immediately above $i$ in $T$, and $c_2'(F) = j(T)$. Now case (c) of $g$ reverses the application of $f$ to $T$. \qedhere
    \end{enumerate}
\end{proof}

We are now ready to prove Theorem \ref{thm:third}.
\begin{proof}[Proof of Theorem \ref{thm:third}]
    The well-defined maps $f: Z_n \setminus Y_n \to Z_n \setminus X_n$ and $g: Z_n \setminus X_n \to Z_n \setminus Y_n$ satisfy $f \circ g = g \circ f = \id$, hence they are bijections. Hence $\lvert Z_n \setminus X_n\rvert = \lvert Z_n \setminus Y_n\rvert$. As $X_n,Y_n \subseteq Z_n$, this implies $\lvert Z_n\rvert - \lvert X_n\rvert = \lvert Z_n\rvert - \lvert Y_n\rvert$, or $\lvert X_n\rvert = \lvert Y_n\rvert$, as desired.
\end{proof}

\begin{remark*}
    Let $g': Y_n \to X_n$ be defined by $g'(T) = g^{k_T}(T)$ for all $T$, where $k_T$ is the number of cells in the third column of $T$. Each application of $g$ decreases the number of cells in the third column by $1$, so $g'$ is well-defined. Similarly, let $f': X_n \to Y_n$ be the map sending $T$ to the tableaux obtained by repeatedly applying $f$ to $T$ until the result lies in $Y_n$. This process terminates, since each application increases the number of cells in the third column. As $f$ and $g$ are inverses, $f'$ and $g'$ are inverses and hence bijections. Therefore, $f'$ and $g'$ are explicit bijections between $X_n$ and $Y_n$.
\end{remark*}

We may now prove Theorem \ref{thm:motzkin}.

\begin{proof}[Proof of Theorem \ref{thm:motzkin}]
    \[\lvert W_2'(n)\rvert \stackrel{\ref{lem:first}}= \lvert X'_n\rvert \stackrel{\ref{lem:second}}= \lvert X_n\rvert \stackrel{\ref{thm:third}}= \lvert Y_n\rvert \stackrel{\text{\cite{motzkinsyt}}}= M_n. \qedhere\]
\end{proof}

\section*{Acknowledgements}

This project was completed at the Research Science Institute (RSI) at MIT by the second author under the mentorship of the first author. We thank AnaMar\'ia Perez, Tanya Khovanova, David Jerison, and Jonathan Bloom for their invaluable feedback. We thank Colin Defant for help formulating this project. We thank Jerry Zhang for useful insights on his stack-sorting tableaux. We thank MIT and the Center for Excellence in Education for making RSI possible. The second author thanks Joseph Mandarino for his sponsorship, making his attendance at RSI possible.

\bibliographystyle{alpha}
\bibliography{references}

\smallskip

\noindent
Ryota Inagaki \\
\textsc{
Department of Mathematics, Massachusetts Institute of Technology\\
77 Massachusetts Avenue, Building 2, Cambridge, Massachusetts, U.S.A. 02139}\\
\textit{E-mail address: }\texttt{inaga270@mit.edu}
\medskip

\noindent
Michael Luo \\
\textsc{
Minnetonka High School \\
18301 MN-7, Minnetonka, Minnesota, U.S.A.  55345}\\
\textit{E-mail address: }\texttt{luomichaelcunhao@gmail.com}
\medskip

\end{document}